\documentclass[12pt]{article}
\usepackage{array}
\usepackage{xcolor}
\usepackage[hidelinks]{hyperref}
\usepackage{amsmath}
\usepackage{amsfonts}
\usepackage{amsthm}
\usepackage{calrsfs}
\usepackage{enumerate}
\usepackage{geometry}

\usepackage{amssymb,latexsym}

\newcommand{\SL}{\mathrm{SL}}

\newcommand{\FF}{\mathbb F}

\newcommand{\SSS}{\mathbb S}
\newcommand{\KK}{\mathbb K}

\newcommand{\cL}{\mathcal L}

\newcommand{\cP}{\mathcal P}

\newcommand{\cC}{\mathcal C}

\newcommand{\PG}{\mathrm{PG}}
\newcommand{\vep}{\varepsilon }
\newcommand{\PGL}{\mathrm{PGL}}

\newcommand{\GL}{\mathrm{GL}}

\newcommand{\Aut}{\mathrm{Aut}\,}
\newcommand{\Tr}{\mathrm{Tr}}
\newcommand{\rank}{\mathrm{rank}\,}

\newtheorem{theorem}{Theorem}[section]
\newtheorem{lemma}[theorem]{Lemma}
\newtheorem{corollary}[theorem]{Corollary}
\newtheorem{prop}[theorem]{Proposition}

\theoremstyle{definition}
\newtheorem{remark}{Remark}

\newtheorem{definition}[theorem]{Definition}
\newcommand{\cH}{\mathcal H}
\newcommand{\cM}{\mathcal M}
\newcommand{\cX}{\mathcal X}

\newcommand{\bee}[1]{\pmb{e}_{#1}}
\newcommand{\bet}[1]{\pmb{\eta}_{#1}}

\newcommand{\segreG}{\Gamma}
\newcommand{\lrootG}{\bar{\Gamma}}

\author{I.~Cardinali, L.~Giuzzi}
\title{Linear codes arising from the point-hyperplane geometry --- Part II: the twisted embedding}
\begin{document}
\maketitle
\begin{abstract}
Let $\lrootG$ be the point-hyperplane geometry of a projective space $\PG(V),$ where $V$ is a $(n+1)$-dimensional vector space over a finite field $\FF_q$ of order $q.$
Suppose that $\sigma$ is an automorphism of $\FF_q$  and consider the projective embedding $\vep_{\sigma}$ of $\lrootG$ into the projective space $\PG(V\otimes V^*)$ mapping the point $([x],[\xi])\in \lrootG$ to the projective point represented by the pure tensor $x^{\sigma}\otimes \xi$, with $\xi(x)=0.$
In~\cite{parte1}, we focused on the case $\sigma=1$ and we studied the projective code arising from the projective system $\Lambda_1=\vep_{1}(\lrootG).$ Here we focus on the case $\sigma\not=1$ and we investigate the linear code $\cC( \Lambda_{\sigma})$
arising from the projective system $\Lambda_{\sigma}=\vep_{\sigma}(\lrootG).$
  In particular, after having verified that
$\cC( \Lambda_{\sigma})$ is a minimal code, we determine its parameters,
its minimum distance as well as its automorphism
group.
We also give a (geometrical) characterization of its minimum and second lowest weight codewords and determine its maximum weight when $q$ and $n$ are
both odd.
\end{abstract}

\section{Introduction}
Let $V$ be a vector space of dimension $k+1$ over a finite field $\FF_q$ of order $q$. It is well known that given a set $\Omega$ of $N$ points spanning $\PG(V)$
(a \emph{projective system of $\PG(V)$}) it is possible to construct a projective code $\cC(\Omega)$ of length $N$
and dimension $k+1$ by taking as generator matrix for $\cC(\Omega)$
a matrix $G$ whose columns are vector representatives of the points
of $\Omega$; see~\cite{TVZ}.
This approach, pioneered  by MacWilliams
in her Ph.D. Thesis~\cite[Lemma 1.5]{MacWilliams1962}, establishes a nice interplay between the properties of the code and the geometry of the projective system itself; for instan9ce, the minimum distance of the code, its higher Hamming weights, its minimality as well as its automorphism group can all be
studied by looking at the geometric properties of $\Omega$.

In a series of papers, we have been interested in studying codes
arising from embeddings of point-line geometries; see~\cite{Cardinali2016b,
  Cardinali2018, Cardinali2018c, parte1, Cardinali2016a}.

In particular, in~\cite{parte1} we considered the linear code arising from a special subvariety $\Lambda_1$
of the Segre variety $\Lambda$ of $\PG(V\otimes V^*)$, where $V^*$ is the dual space of $V$, consisting of all pure tensors $x\otimes \xi\in V\otimes V^*$ such that $\xi(x)=0.$

The Segre variety  is an important algebraic variety
arising from the Segre embedding $\vep$
of the Segre geometry
$\segreG$ of $\PG(V\otimes V^*).$ Accordingly, $\Lambda=\vep(\Gamma).$ The variety $\Lambda$ has been extensively studied, as it corresponds to
the determinantal variety consisting of all matrices
of rank $1.$
The code arising from the projective system of the $\FF_q$-rational
points of $\Lambda$ has been
investigated by Beleen, Ghorpade and Hasan in 2015
\cite{BGH}.

In~\cite{parte1} we introduced and studied  a family of codes associated to the projective system $\Lambda_1$ defined by the image under the  \emph{Segre embedding} of the point-hyperplane geometry of $\PG(V).$
Keeping the same notation as in~\cite{parte1}, we denote by $\lrootG$ the  point-hyperplane geometry of $\PG(V).$ Then $\lrootG$ is a {\it geometric  hyperplane} of $\segreG$ and the \emph{Segre embedding} $\bar{\vep}$ of $\lrootG$ maps any point $([x],[\xi])$ of $\lrootG$ into the point $[x\otimes \xi]\in \PG(M^0_{n+1}(q)),$ where $M^0_{n+1}(q)$ is  a hyperplane of $V\otimes V^*.$ According to this notation, we have $\Lambda_1:=\bar{\vep}(\lrootG).$

This paper can be considered as the natural continuation of~\cite{parte1}:
now we investigate the linear code arising from another relevant subvariety $\Lambda_{\sigma}$ of the Segre variety defined again as the image of the point-hyperplane geometry $\lrootG$ but under the so-called {\it twisted embedding} of $\lrootG$, which turns out to be  a projective embedding not isomorphich to the Segre embedding (if $\sigma\not=1$).

The existence of the twisted embedding assumes the existence of non-trivial automorphisms of $\FF_q$ (see Section~\ref{prelim}) and its definition is a slight modification of the definition of the Segre embedding. Indeed, suppose $\sigma\in\Aut(\FF_q),\,\sigma\not= 1.$ Then the \emph{twisted embedding} $\bar{\vep}_{\sigma}$ of $\lrootG$ maps any point $([x],[\xi])$ of $\lrootG$ into the point $[x^\sigma\otimes \xi]\in \PG(V\otimes V^*).$ When $\sigma=1$, $\vep_1$ is just the Segre embedding. Accordingly,
\[\Lambda_{\sigma}:=\bar{\vep}_{\sigma} (\lrootG)=\{[x^\sigma\otimes \xi]\colon  [x]\in\PG(V), [\xi]\in\PG(V^*) \text{ and } [x]\in [\xi]\}\]
In the present paper  we study the code $\cC(\Lambda_{\sigma})$
arising from $\Lambda_{\sigma}.$

\subsection{Main results}
In this paper,
after introducing the family of codes $\cC(\Lambda_{\sigma})$ arising from the projective system $\Lambda_{\sigma}$ defined in the Introduction, we will determine their parameters and  their  automorphism group.
We will also prove that $\cC(\Lambda_{\sigma})$ is a minimal code and characterize its minimum and second lowest weight codewords.
More in detail, the main results  are the following.
\begin{theorem} \label{main thm 3}
 Suppose $V$ is a $(n+1)$-dimensional vector space over $\FF_q$ and $1\not=\sigma\in\Aut(\FF_q).$ Let $\Lambda_{\sigma}$ be the projective system of $\PG(V\otimes V^*)$ whose points are represented by pure tensors $x^\sigma \otimes \xi$ such that $\xi(x)=0.$ The  $[N_{\sigma},k_{\sigma},d_{\sigma}]$-linear code $\cC(\Lambda_{\sigma})$ associated  to $\Lambda_{\sigma}$ has parameters
\[N_{\sigma}=\frac{(q^{n+1}-1)(q^n-1)}{(q-1)^2},\quad k_{\sigma}=n^2+2n+1, \]
\[ d_{\sigma}=\begin{cases}
	q^3-\sqrt{q}^3 & \text{ if }  \sigma^2=1 \text{ and } n=2, \\
	q^{2n-1}-q^{n-1} & \text{ if }  \sigma^2\neq 1 \text{ or } n>2.
\end{cases}
\]
\end{theorem}
Suppose $a\in\FF_q$ and $\sigma\in\Aut(\FF_q)$.
Denote by $\KK=\FF_s$ the fixed field of $\sigma$. 
When $\sigma\neq1$ and $\sigma^2=1$ we put $s=\sqrt{q}$ and we define the {norm} of $a$
as $N(a):=a^{s+1}$.

In the next theorem we characterize the minimum, the second lowest and some of the maximum weight codewords in terms of geometrical hyperplanes of
the geometry $\lrootG.$

We refer to Section~\ref{prelim} for the concept of {\it geometrical hyperplane of $\lrootG$ arising from an embedding} and to Section~\ref{hyperplanes} for the definition and description of the hyperplanes of $\lrootG$ mentioned in the following theorem.
\begin{theorem}
  \label{main thm 4}
	Let $\cC(\Lambda_{\sigma})$ be the code as defined in Theorem~\ref{main thm 3}.  The following hold.
        \begin{enumerate}
        \item\label{c4t0} The code $\cC(\Lambda_{\sigma})$ is minimal.
  \item\label{c4t1}
    If $n=2$ and $\sigma^2=1$, $\sigma\neq1$
    then the minimum weight codewords of $\cC(\Lambda_{\sigma})$
    have weight $q^3-\sqrt{q}^3$ and
    correspond
    to the hyperplanes of $\lrootG$ associated to matrices $M$ for which there exist three linearly independent vectors
    $\xi_1,\xi_2,\xi_3
    \in V^*$ and
    $\alpha,\beta,\gamma\in\FF_q$ with $N(\alpha)=N(\beta)=N(\gamma)$
    such that
    \[ \xi_1M=\alpha\xi_1^{\sigma},\qquad
      \xi_2M=\beta\xi_2^{\sigma},\qquad
      \xi_3M=\gamma\xi_3^{\sigma}.
    \]
   \item\label{c4t2}
  If $n>2$ or $\sigma^2\neq1$ then the minimum weight codewords of $\cC(\Lambda_{\sigma})$
    have weight $q^{2n-1}-q^{n-1}$ and  correspond
    to quasi-singular, non-singular hyperplanes of $\lrootG.$ The second lowest weight codewords of $\cC(\Lambda_{\sigma})$ have weight
    $q^{2n-1}$ and  correspond to singular hyperplanes
    of $\lrootG$.
  \item\label{pt4-13} If both $q$ and $n$ are odd, then the maximum weight codewords of $\cC(\Lambda_{\sigma})$
    have weight $q^{n-1}(q^{n+1}-1)/(q-1)$. Every semi-standard spread type hyperplane of $\lrootG$ (in
    particular $\sigma^2=1$ and $n$ is odd) is associated to a maximum weight codeword of $\cC(\Lambda_{\sigma})$. 
   \end{enumerate}
\end{theorem}

It is interesting to point out that the minimum and the second lowest weight codewords of $\cC(\Lambda_{\sigma})$ are associated to the same families of hyperplanes of $\lrootG$ as the minimum and the second lowest weight codewords of the code $\cC(\Lambda_1)$ studied in~\cite{parte1}; likewise the
argument leading to the minimality of $\cC(\Lambda_{\sigma})$ depend
only on $\lrootG$.
This suggest that  the properties of the geometry   $\lrootG$ play a crucial role in giving information on the structure of the codes associated to it, regardless of
the way $\lrootG$ is embedded.

We refer to the beginning of Section~\ref{twisted code} for a description of the codewords $c_M$ of $\cC(\Lambda_{\sigma}).$ 
\begin{theorem}
  \label{main thm 5}
  The group $\GL(n+1,q)$ acts as an automorphism group of the
  code $\cC(\Lambda_{\sigma})$ via the action
  \[ \varrho(g): c_M\to c_{g^{-1}Mg^{\sigma}}. \]
  The kernel of this action is the subgroup of all scalar matrices
  of the form $K_{\sigma}:=\{ \alpha I: \alpha\in\FF_s \}$.


\end{theorem}

\subsection{Organization of the paper}
Since we are considering codes arising from  different embeddings of the same geometry $\lrootG$,  this work can be regarded as a continuation of~\cite{parte1}.
Consequently,
some of the set-up (as well as the notation) of the present paper
 overlaps with those of~\cite{parte1}, even if the specific
 arguments and results about the code turn out to be quite different.
 In Section~\ref{Sec 1} we shall limit the presentation of the
 topics otherwise
 introduced in~\cite[Section 2]{parte1}, without providing proofs
 or extensive
 discussion (for which we refer the reader to the previous paper).

The organization of the paper is as follows.
In Section~\ref{notation} we set the notation to be used throughout
the paper and recall the notion of \emph{saturation form} for
the space $M_{n+1}(q)$. In Section~\ref{prelim} we  recall the basics about
point-line geometries and their embeddings; in Section~\ref{pmcodes}
we recall the basics about codes from projective systems and minimal
codes; we refer to~\cite{survey,parte1} for more details.
Sections~\ref{long root geometry} and~\ref{hyperplanes} are dedicated to the
point-hyperplane geometry and its hyperplanes; we also discuss here
its natural and its twisted embedding. In Section~\ref{twisted code} we will prove our main theorems.
Theorem~\ref{main thm 3} will be proved in Subsection~\ref{proof of main thm 3}, Theorem~\ref{main thm 4} will be proved in Subsection~\ref{proof of main thm 4} and Theorem~\ref{main thm 5} will be proved in Subsection~\ref{proof of main thm 5}.

\section{Notation and Basics}\label{Sec 1}
\subsection{Notation}
\label{notation}
Let $V=V(n+1,\FF_q)$ be a $(n+1)$-dimensional vector space over $\FF_q$ and $V^*$ its dual. Henceforth we always assume that $E = (\bee{i})_{i=1}^{n+1}$ is a given fixed basis of $V$
and $E^*=(\bet{i})_{i=1}^{n+1}$ is its dual basis in $V^*$.
Vectors will always be regarded as expressed by their components
with respect to these bases. We also adopt the convention that
the elements of $V$ are denoted by roman letters
and are column vectors, while the elements of $V^*$ are
row vectors denoted by Greek letters.

Given an automorphism $\sigma$ of $\FF_q$ and $x=\sum_{i=1}^n\bee{i}x_i\in V$,
$\xi=\sum_{i=1}^n\bet{i}\xi_i\in V^*$, we put
$x^{\sigma}=\sum_{i=1}^n\bee{i}x_i^{\sigma}$ and
$\xi^{\sigma}=\sum_{i=1}^n\bet{i}\xi_i^{\sigma}$.


It is well known that the space $V\otimes V^*$ can be canonically
identified with the space $M_{n+1}(q)$ of all $(n+1)\times(n+1)$ matrices with
entries in $\FF_q$.
More in detail,
the elements $E\otimes E^*=\{\bee{i}\otimes \bet{j}\}_{1\leq i,j\leq n+1}$ form a basis of $V\otimes V^*$. If we map $\bee{i}\otimes\bet{j}$ to the elementary matrix $\bee{ij}$ whose only non-zero entry is $1$ in position $(i,j)$, we
see that this gives the isomorphism $\phi\colon V\otimes V^* \rightarrow M_{n+1}(q)$ described by
\[ x\otimes\xi \in V\otimes V^*\mapsto \begin{pmatrix} x_1\\\vdots\\x_{n+1} \end{pmatrix}
  \begin{pmatrix} \xi_1 & \dots & \xi_{n+1} \end{pmatrix}\in M_{n+1}(q). \]
Thus, the tensor $x \otimes \xi\in V\otimes V^*$ can be
regarded as the usual matrix
product  $x \xi$ between a column and a row vector and  for a matrix $M \in M_{n+1}(q)$, the product $\xi M x$ is the scalar obtained as the product of the row $\xi$ times
$M$ times the column $x$.
Recall that a matrix $M$ can be written as $x\xi$ (namely, it is
a pure tensor of the form $x\otimes\xi$) if and only if it has rank $1$.

If $M:=\sum_{i,j=1}^{n+1} m_{ij}(\bee{i}\otimes \bet{j})$ is
any element of $V\otimes V^*$, then the \emph{trace} of $M$ is defined
as $\Tr(M):=\sum_{i,j=1}^{n+1}m_{ij}\bet{i}(\bee{j})=\sum_{i=1}^{n+1}m_{ii}$.
It can be seen that the trace is independent from the choice of
the bases $E$ and $E^*$; indeed, for pure tensors $x\otimes\xi$ the
trace of the corresponding matrix corresponds to the value of $\xi(x)$.

For any two matrices $X,Y\in M_{n+1}(q)$ we can consider \emph{the saturation form}
$f \colon  M_{n+1}(q) \times  M_{n+1}(q) \rightarrow \FF_q$ of $M_{n+1}$
 given by
\begin{equation}\label{sat form}
  f(X, Y ) = \Tr(XY), \,\,\, \forall\,\, X, Y \in  M_{n+1}(q),
\end{equation}
where $XY$ is the usual row-times-column product.
It is immediate to see that this form is bilinear, symmetric and non-degenerate;
furthermore it does not depend on the choice of the basis of $V$; for more
detail see~\cite{Pasini24}.

Let $\perp$ be the orthogonality relation associated to $f.$ Since $f$ is non-degenerate, the hyperplanes of $M_{n+1}(q)$ are the orthogonal spaces
\[M^{\perp}=\{X\in M_{n+1}(q)\colon \Tr(XM)=0 \},\]
 for $M \in  M_{n+1}(q) \setminus\{O\}$ and,
for two matrices $M,N \in M_{n+1}(q) \setminus\{O\}$, we have $M^{\perp} = N^{\perp}$ if and only
if $M$ and $N$ are proportional.


The pure tensors in the orthogonal space of a matrix $M$ with
respect to the saturation form admit a simple description,
as illustrated by the following proposition; see~\cite{parte1}
for the proof.
\begin{prop}\label{prop}
Let $x \in V \setminus  \{O\}$, $\xi \in V^* \setminus  \{O\}$  and $M \in  M_{n+1}(q)$. Then
$x \otimes  \xi \in M^{\perp}$ if and only if $\xi M x = 0.$
\end{prop}

Henceforth, we shall adopt either
the matrix notation or the tensor product notation, silently switching
between them, according to convenience.

Turning to projective spaces, let $\PG(n,q)=\PG(V)$ be the $n$-dimensional projective space defined by $V.$
When we need to distinguish between a non-zero vector $x$ of $V$ and the point
of $\PG(V)$ represented by it, we denote the latter by $[x]$. We extend
this convention to subsets of $V.$ If $X \subseteq V \setminus \{0\}$ then $[X] := \{[x] | x \in X\}$. The
same conventions will be adopted for vectors and subsets of $V^*$ and $V \otimes V^*$. In
particular, if $\xi \in V^* \setminus \{0\}$ then $[\xi]$ is the point of $\PG(V^*)$ which corresponds
to the hyperplane $[\ker(\xi)]$ of $\PG(V )$. In the sequel we shall freely take $[\xi]$ as a
name for $[\ker(\xi)].$ Accordingly, if $0\in V^*$ then $[0]:=\PG(V).$

Observe that in terms of projective spaces,
Proposition~\ref{prop} states that
the point $[x \otimes  \xi]\in\PG(V\otimes V^*)$ is contained in
the space
$[M^{\perp}]$ if and only if the point $[x]\in\PG(V)$ is contained
in $[\xi M]\in\PG(V^*)$.

\subsection{Embeddings and hyperplanes of point-line geometries }
\label{prelim}

Let $\Gamma=(\cP,\cL)$
and $\Gamma'=(\cP',\cL')$ be two point-line geometries with
pointsets $\cP$ respectively  $\cP'$ and linesets $\cL$ respectively $\cL'$
and incidence given by inclusion.
An \emph{embedding} of $\Gamma$ in $\Gamma'$ is an injective function
$\iota:\cP\to\cP'$ which maps lines in $\cL$ onto lines in
$\cL'$.

When $\Gamma'=\PG(V)$ is a projective geometry, we say that
$\vep:\Gamma\to\PG(V)$ is a \emph{projective embedding} of $\Gamma$ if
  $\vep$ is an embedding in the sense mentioned above and
  the image of $\vep$ spans $\PG(V)$.
  In this case we call \emph{dimension of $\vep$} the vector dimension
  of $V$.


A \emph{subspace} of  $\Gamma$ is a nonempty subset $\cX$ of  $\cP$ such
that every line $\ell\in\cL$ meeting $\cX$ in at least $2$ points
is fully contained in $\cX$. A subspace of $\Gamma$ is {\it maximal} if it is not properly contained in any other proper subspace of $\Gamma.$
A proper subspace $\cX$ of $\Gamma$ is
a {\it geometric hyperplane} if it meets every line in at least
$1$ point.

Given a projective embedding $\vep\colon \Gamma \rightarrow \PG(V)$
of $\Gamma$, for any projective hyperplane $W$ of $\PG(V)$, the point set $\mathcal{W}:=\vep^{-1}(W)$ is a geometric hyperplane of $\Gamma$ called  {\it the geometric hyperplane
 arising from  $W$ via $\vep$}.
Clearly, $\vep( \vep^{-1}(W)) = W\cap  \vep(\cP)$.

\subsection{Projective and minimal codes}
\label{pmcodes}
We recall that given a projective system $\Omega$ consisting of
$N$ points in $\PG(k-1,q)$, the projective code $\cC(\Omega)$ defined by
$\Omega$ is a $[N,k,d]$-linear code whose generator $G$ matrix consists of
the vector representatives of the points of $\Lambda$.

The minimum distance of $\cC(\Omega)$ is related to the maximum possible intersection of the projective system $\Omega$ with any hyperplane of
$\PG(k-1,q)$. In particular,
\[ d=N-\max_{H\in\PG(k-1,q)^*} |H\cap\Omega|. \]

Another important property of codes which has a nice geometrical counterpart is that of {\it minimality}.
We briefly recall the definition of minimal codes.

Let $\cC$ be a projective $[N,k,d]$-code. For any $c=(c_1,\dots,c_N)\in\cC$ the \emph{support} of $c$ is the set
$\mathrm{supp}(c)=\{i: c_i\neq 0\}$.

The notion of \emph{minimal codewords} has been introduced by Massey
in~\cite{Massey}.

\begin{definition}
  \label{minc}
A codeword $c\in\cC$ is \emph{minimal} if
  \[ \forall c'\in\cC:\mathrm{supp}(c')\subseteq\mathrm{supp}(c)
    \Rightarrow\exists\lambda\in\FF_q: c'=\lambda c. \]
  A code $\cC$ is \emph{minimal} if all its codewords are minimal.
\end{definition}

Obviously, all codewords with minimum weight are minimal,  however, to determine if \emph{all} codewords
satisfy Definition~\ref{minc} might be a difficult problem in general.

Minimal codes have been extensively investigated by Ashikhmin and Barg~\cite{Ashikhmin1998} who also provided a necessary condition on the weights of the codewords for the code to be minimal.

For projective codes arising from a projective system $\Omega$,  being minimal is equivalent
to ask that the projective system $\Omega$ is a so-called \emph{cutting blocking set with respect to hyperplanes}, i.e.  for any hyperplane $H$ of $\PG(\langle\Omega\rangle)$,
$ \langle H\cap\Omega\rangle=H$;
see~\cite{Alfarano2022a,Alfarano2022,Bonini2021, Davydov2011,Tang2021}.

Relying on the notion of cutting sets and on the properties of the  geometric hyperplanes of a point-line geometry $\Gamma$,  we have obtained in~\cite{survey} a sufficient condition for a code to be minimal
in terms of the maximality of the geometrical hyperplanes of $\Gamma$. We shall make use of the following characterization from~\cite{survey}.

\begin{prop}
	\label{c:min}
	Suppose that $\Gamma=(\cP,\cL)$ is a point-line geometry where every
	geometric hyperplane is a maximal subspace.
	Then the projective code
	$\cC(\vep(\Gamma))$ is minimal, for any projective embedding
	$\vep$ of $\Gamma$.
\end{prop}

 \medskip

A \emph{monomial transformation} of $\FF_q^N$ is an
invertible linear transformation $\FF_q^N\to\FF_q^N$ which
is described with respect to the canonical basis by the product of
a permutation matrix $P$ by a diagonal matrix $D$.
By the MacWilliams extension theorem~\cite[Theorem 1.10]{MacWilliams1962}, the group of monomial transformations of $\FF_q^N$ extends to the same
as the group of isometries of $\FF_q^N$ with respect to Hamming distance.

An \emph{automorphism} of the code $\cC(\Omega)$ (thus regarded as a subspace of $\FF_q^N$) is
an isometry of $\FF_q^N$ which maps codewords into codewords; see~\cite{MS}.
It can be shown that if $\gamma$ is an automorphism of the geometry
$\Gamma$ which lifts into a linear transformation of $\PG(\langle\vep(\Gamma)\rangle)$
via an embedding $\vep$, then $\gamma$ induces an automorphism of
$\cC(\vep(\Gamma))$.

\subsection{The point-hyperplane geometry $\lrootG$ of $\PG(V)$ and its embeddings}\label{long root geometry}
Following the notation of~\cite{parte1}, denote by $\Gamma$  the Segre geometry $\PG(V)\otimes \PG(V^*)$ whose points are all the  ordered pairs $([p],[\xi])$, where $[p]$ and $[\xi]$ are respectively a point and a hyperplane of $\PG(V)$ and $\vep\colon \Gamma \rightarrow \PG(V\otimes V^*)$ denote the Segre embedding of $\Gamma$, mapping $(x,\xi)$ to $[x\otimes\xi]$.

 As mentioned in Section~\ref{notation}, we shall always
silently identify $V\otimes V^*$ with $M_{n+1}(q)$ by the isomorphism
induced by
$(\bee{i}\otimes \bet{j})\to \bee{i}\bet{j}=:\bee{ij}$.

The \emph{point-hyperplane geometry} of $\PG(V)$ is the
geometry $\lrootG=(\cP,\cL)$ whose points are all the pairs
$(p,H)\in\PG(V)\otimes\PG(V^*)$ such that $p\in H$ and whose
lines are either of the form
$\ell_{r,H}:=\{ (p,H): p\in r \}$, where $r$ is a line of $\PG(V)$, or $\ell_{p,S}:=\{ (p,H) : S\subseteq H \}$
 where  $S$ is a subspace of $\PG(V)$ of
codimension $2$ (i.e. a line of $\PG(V^*)$).
Its linear automorphism group is
$\PGL(n+1,q)$ and it acts transitively on the points of $\lrootG$.

The geometry $\lrootG$ is also known as
the \emph{long root geometry for the
  special linear group $\SL(n+1,\FF_q)$}; see e.g.
~\cite{ILP24,Kantor1979}.

It follows from the definition of $\lrootG$ that the identity map
$\iota:\lrootG\to\segreG$ sending the point $(x,\xi)\in \lrootG$ to the same point $(x,\xi)\in \segreG$, is
an embedding of geometries.

Consequently, the map $\bar{\vep}:=\vep\circ \iota$ is a projective
embedding of $\lrootG$, called the \emph{Segre} or \emph{natural embedding} of $\lrootG$.
Note that the image of $\bar{\vep}$ spans $\PG(M^0_{n+1}(q))$
where $M^0_{n+1}(q)$ is the hyperplane of the traceless matrices of $M_{n+1}(q).$
In~\cite{parte1} we studied the code $\cC(\Lambda_1)$ where
$\Lambda_1:=\bar{\vep}(\lrootG).$

The geometry $\lrootG$ turns out to be a geometric hyperplane of the Segre geometry $\Gamma$ which is called in~\cite{Hvm24} of {\it black type.} It is shown in~\cite{Hvm24} that any geometric hyperplane of $\Gamma$
of black type corresponds to an embedding of $\lrootG$ in $\Gamma$ and
that these hyperplanes might not lie in the same orbit with respect to
the collineation group of the Segre geometry. Actually, $\lrootG=\iota(\lrootG)\subset \Gamma$ is one of these hyperplanes.

Let $\sigma$ be a non-trivial automorphism of $\FF_q.$ Consider the {\it twisted map} of $\lrootG$ defined as follows
\[\iota_{\sigma}:\lrootG\to\Gamma,\,\,(x,\xi)\mapsto (x^{\sigma},\xi).\]
Then, see~\cite{Hvm24}, $\iota_{\sigma}(\lrootG)\cong\lrootG$ is again a hyperplane of $\Gamma$ and different automorphisms of $\FF_q$ correspond to different orbits of geometric hyperplanes of $\Gamma$ under the automorphism group of $\Gamma.$
Define the {\it twisted embedding} of $\lrootG$ as
\[\bar{\vep}_{\sigma}=\vep\circ\iota_{\sigma}.\]
More explicitly, we have
 \[\bar{\vep}_{\sigma} \colon \lrootG \rightarrow \PG(M_{n+1}(q)),\,\,\bar{\vep}_{\sigma} (([x], [\xi]))=[x^{\sigma}\otimes \xi],\]
where $x^{\sigma}:=({x_i}^{\sigma})_{i=1}^{n+1}.$
The dimension of $\bar{\vep}_1:=\bar{\vep}$ is $(n+1)^2-1$, while
the dimension of $\bar{\vep}_{\sigma}$ is $(n+1)^2$, since
the image of $\lrootG$ by
means of $\bar{\vep}_{\sigma}$ spans $\PG(V\otimes V^*)$.
We put
\begin{equation}
	\label{e7}
	\Lambda_{\sigma}:=\bar{\vep}_{\sigma} (\lrootG)=\{[x^\sigma\otimes \xi]\colon  [x]\in\PG(V), [\xi]\in\PG(V^*) \text{ and } [x]\in [\xi]\}
\end{equation}

For more information on the embeddings of $\lrootG$, we refer the reader
to~\cite{Pas24}. In particular it is shown that embeddings related
to different automorphisms of $\FF_q$ are inequivalent. This being
said, most of the properties of the code $\cC(\Lambda_{\sigma})$
which will be studied in the present paper depend only on whether
$\sigma=1$ (in which case we have the Segre embedding and we refer
the reader to~\cite{parte1}) or $\sigma\neq1$.

The group $\GL(n+1,q)$ acts on the geometry $\lrootG$ as an automorphism
group by the following  action:
given $([x],[\xi])\in\lrootG$ and $g\in\GL(n+1,q)$, then
\[ ([x],[\xi])\to ([x],[\xi])^g:=([gx],[\xi g^{-1}]). \]
The kernel of the action consists exactly of the scalar
matrices; so $\PGL(n+1,q)=\GL(n+1,q)/\{\alpha I: \alpha\in\FF_q \}$
acts faithfully as a permutation group on $\lrootG$.
As $\PGL(n+1,q)$ is flag-transitive on $\PG(V)$, its action
on $\lrootG$ is transitive.
As the embedding $\bar{\vep}_{\sigma}$ is {\it homogeneous},
this action of $\GL(n+1,q)$ lifts through $\bar{\vep}_{\sigma}$ to an automorphism subgroup of $\PG(M_{n+1}(q))$ as follows
\begin{multline}
  \label{action-gl}
\bar{\vep}_{\sigma}( ([x],[\xi])^g)=
\bar{\vep}_{\sigma}( ([gx],[\xi g^{-1}]))=[g^{\sigma}x^{\sigma}\otimes \xi g^{-1}]= \\
[ g^{\sigma}(x^{\sigma}\xi)g^{-1}]=
g^{\sigma}\bar{\vep}_{\sigma}( ([x],[\xi] ) g^{-1}.
\end{multline}
In particular, extending the action to all of $M_{n+1}(q)$ we put
\[ M^g:=g^{\sigma}Mg^{-1},\qquad\qquad \forall M\in M_{n+1}(q). \]
Write  $\Lambda_{\sigma}=\{ [X_1],\dots,[X_N]\}$.
By definition~(\ref{action-gl}) of the action, $\Lambda_{\sigma}^g=\Lambda_{\sigma}$ for
all $g\in\GL(n+1,q)$; so
there is a group homomorphism
$\pi:\GL(n+1,q)\to S_N$ (where $S_N$ is the symmetric group on $\{1,\dots,N\}$) such that
\begin{equation}
  \label{action-L}
  [X_i]^g=[X_i^g]=[g^{\sigma}X_ig^{-1}]=[X_{\pi(g)(i)}].
\end{equation}
Observe that this action of $\GL(n+1,q)$ on $M_{n+1}(q)$ is matrix conjugation
if and only if $\sigma=1$.

\subsection{Hyperplanes of $\lrootG$}\label{hyperplanes}
In this section we will briefly recall from~\cite{Pas24} and~\cite{Pasini24} the most significant results related to the hyperplanes of $\lrootG$ arising from the embeddings $\bar{\vep}$ and $\bar{\vep}_{\sigma}.$
For the case of the embedding $\bar{\vep}$ see also Section 2.7 of~\cite{parte1}.

We begin with a general theorem about geometric hyperplanes.
\begin{theorem}[Theorem 1.5, \cite{Pasini24}]
\label{msub}
  All hyperplanes of $\lrootG$ are maximal subspaces.
\end{theorem}

Take now $M\in M_{n+1}(q)\setminus\langle I\rangle$ and let $\bar{\vep}$ be the Segre embedding of $\lrootG.$ Then
\[{\cH}_M:=\bar{\vep}^{-1}([M^{\perp}])
  =\bar{\vep}^{-1}(\{ [X]\in \PG(M^0_{n+1}(q))\colon \Tr(XM)=0 \})\]
is a geometric hyperplane of $\lrootG$ called a \emph{hyperplane of plain type}, as defined in~\cite{Pasini24}.
\begin{prop}\cite[Corollary 1.7]{Pasini24}
  The hyperplanes of $\lrootG$ which arise from the Segre embedding $\bar{\vep}$ are precisely those of plain type.
\end{prop}
If $\sigma\in \Aut(\FF_q)$ then
\[{\cH}_{M,\sigma}:=\bar{\vep}_{\sigma}^{-1}([M^{\perp} \cap \bar{\vep}_{\sigma}(\lrootG)])
  =\bar{\vep}_{\sigma}^{-1}(\{ [X]\in \PG(M_{n+1}(q))\colon \Tr(XM)=0 \})\]
is again a hyperplane of $\lrootG$, but it is in general different from
$\cH_M$; see Proposition~\ref{pp} for more details.

Take $p\in\PG(V)$, $A\in\PG(V^*)$. Put ${\cM}_p:=\{ (p, H) : p\in H \}$
and ${\cM}_A:=\{ (x, A) : x\in A \}$. Then,
\begin{equation}\label{quasi-singular}
 \cH_{p,A}:=\{ (x,H) \colon (x,H) \text{ collinear with a point of }
  \cM_p\cup\cM_A \}
  \end{equation}
is a geometric hyperplane of $\lrootG$, called the \emph{quasi-singular hyperplane}
defined by $(p,A)$. If $p\in A$, then $\cH_{p,A}$ is called \emph{singular
  hyperplane of deepest point $(p,A)$} and consists of all points of
$\lrootG$ not at maximal distance from $(p,A)$ in the collinearity
graph of $\lrootG$.

The following theorem further describes the quasi-singular hyperplanes
of $\lrootG$.
\begin{prop}[\S 1.3, \cite{Pasini24}]
\label{qsh}
  Take $[x]\in\PG(V)$ and $[\xi]\in\PG(V^*)$. Then
  the quasi-singular hyperplane $\cH_{[x],[\xi]}$
  is the  hyperplane of
  plain type $\cH_M$ where $M=x\xi$.
\end{prop}
In particular,
all quasi-singular hyperplanes are hyperplanes of plain type arising
from matrices $M$ of rank $1$ and, conversely, for each matrix $M\in M_{n+1}(q)$
of rank $1$ the hyperplane of plain type ${\cH}_M$ is quasi-singular.

\begin{prop}[Theorem 1.6, \cite{Pas24}]
  \label{pp}
  Let $\cH$ be a geometric
  hyperplane of $\lrootG$.
  \begin{enumerate}
  \item If $\cH$ is quasi-singular, then $\cH$ arises
    from $\bar{\vep}_{\sigma}$ for all $\sigma\in\Aut(\FF_q)$.
  \item If $\cH$ is not quasi-singular, then $\cH$ arises
    from $\bar{\vep}_{\sigma}$ for at most one $\sigma\in\Aut(\FF_q)$.
  \end{enumerate}
\end{prop}
By Proposition~\ref{qsh},  there is a one-to-one correspondence between
quasi-singular hyperplanes of $\lrootG$
and proportionality classes of matrices of rank $1$;
Proposition~\ref{pp} shows that these hyperplanes are the only ones
which arise from both the Segre embedding $\bar{\vep}$ and
also from all the twisted embeddings $\bar{\vep}_{\sigma}$; as we will see in Subsection~\ref{sec pesi}, they induce words with the same weight on the code $\cC(\Lambda_{\sigma})$
and on the code $\cC(\Lambda_1)$ studied in~\cite{parte1}. 

We now consider a further family of hyperplanes.
Suppose $\phi$ is a non-linear, fixed-point-free  collineation of $\PG(V).$ 
The set $S_{\phi}:=\{ \langle p,\phi(p)\rangle \colon p\in\PG(V) \}$ is
a line-spread of $\PG(V)$ if and only if $\phi$ is an involution;
see~\cite[Lemma 2.9]{Pasini24}.
Such an involution may exist only
if $n$ is odd, since otherwise $\PG(V)$ does not admit
line-spreads at all.
We shall call a spread $S_{\phi}$ obtained in this way
a \emph{semi-standard line-spread} of $\PG(V)$.
Under our assumptions $\phi$ is an involution, hence
the semi-linear mapping $f\colon V\rightarrow V$
associated to $\phi$ is defined as $f(x)=Mx^{\sigma}$ for $M\in M_{n+1}(q)$ such that $M^{\sigma}=M^{-1}$ and
$\sigma^2=1.$

We say that a line-spread $S$ \emph{admits a dual} if there exists a line spread $S^*$ of
$\PG(V^*)$ such that for every line $\ell^*\in S^*$ (i.e.\ for every $2$-codimensional subspace of $\PG(V)$),  the members of $S$ contained in $\ell^*$, form a line spread of $\ell^*$; see~\cite{Pasini24}.

If a dual spread exists, then this is unique.

Given a spread $S$ admitting a dual $S^*$ it is possible to
construct a geometric hyperplane $\cH_S$ of $\lrootG$;
see~\cite[Theorem 1.11]{Pasini24}.
In particular, if $S_{\phi}$ is a semi-standard line-spread of
$\PG(V)$, then by~\cite[Proposition 2.10]{Pasini24}, $S_{\phi}$
admits a dual spread and the corresponding geometric hyperplane of $\lrootG$ is defined as
\begin{multline*}
   \cH_{\phi}:=\{ ([x],[\xi])\in\lrootG\colon [\xi]\supset \ell_x\}=
   \{ ([x],[\xi])\in\lrootG\colon [x],[\phi(x)]\in [\xi] \}=\\
   \{ ([x],[\xi])\colon \xi x=0, \xi Mx^{\sigma}=0 \},
 \end{multline*}
 where $\ell_{x}:=\langle x,\phi(x)\rangle$ is the unique element of $S$ containing $[x]$ and $M$ is
 the matrix appearing in the definition of $\phi$; see~\cite[Theorem 1.11]{Pasini24}.
We call $  \cH_{\phi}$ a hyperplane of $\lrootG$ of {\it semi-standard spread type}.

 Let now $\bar{\vep}_{\sigma}$ be the twisted embedding of $\lrootG$ obtained
 using the same automorphism $\sigma$ as above.
 The hyperplane $\cH_{\phi}$ of semi-standard spread type arises from the
 twisted embedding $\bar{\vep}_{\sigma}$. Indeed,
 \begin{multline*}
   \bar{\vep}_{\sigma}(\cH_{\phi})=
   \{ \bar{\vep}_{\sigma}(([x],[\xi]))\colon \xi x=0, \xi M x^{\sigma}=0 \}
   = \\
   \{ [x^{\sigma}\otimes \xi]: \xi x=0, \xi Mx^{\sigma}=0 \}=
   \bar{\vep}_{\sigma}(\lrootG)\cap\{ [x^{\sigma}\otimes \xi] \colon \xi M x^{\sigma}=0\}=\\
   \bar{\vep}_{\sigma}(\lrootG)\cap\{ [x\otimes \xi] \colon \xi Mx=0\}=
   \bar{\vep}_{\sigma}(\lrootG)\cap [M^{\perp}].
 \end{multline*}
 
Since $\cH_{\phi}$
 is not a (semi-)singular hyperplane, it arises \emph{only} from
 $\bar{\vep}_{\sigma}$; see~\cite[Theorem 1.12]{Pas24}.
 We summarize the above arguments in the following proposition.
 \begin{prop}
 If a semi-standard spread type hyperplane $\cH_{\phi}$ of $\lrootG$ arises from the
 twisted embedding $\bar{\vep}_{\sigma}$ then $\sigma^2=1$, $n$ is odd and  
   $\phi\colon \PG(V)\rightarrow \PG(V)$ is an involutory
     fixed-point free
   semilinear collineation of the form $\phi([x])=[Mx^{\sigma}]$ with
   $M^{\sigma}=M^{-1}.$ Furthermore
   $\cH_{\phi}=\bar{\vep}_{\sigma}^{-1}(M^{\perp})$.
 \end{prop}
 By the definition of collinearity in $\lrootG$ and a direct counting
 argument in the case of semi-standard spread type hyperplanes, we have the following.
\begin{prop}[Proposition 2.11, \cite{parte1}]
  \label{cardinalita' iperpiani quasi singolari}
The following hold.
\begin{enumerate}
	\item The cardinality of the singular hyperplanes of $\lrootG$ is
          \begin{equation}\label{cc1}
            \frac{(q^{n+1}-1)(q^{n-1}-1)}{(q-1)^2}+\frac{q^n-1}{q-1} q^{n-1}.
          \end{equation}
	\item
	The cardinality of the quasi-singular but not singular hyperplanes of $\lrootG$ is
	\begin{equation}\label{cc2}
          \frac{(q^{n+1}-1)(q^{n-1}-1)}{(q-1)^2}+(\frac{q^n-1}{q-1} +1)q^{n-1}.
        \end{equation}
      \item
        The cardinality of a semi-standard spread type hyperplane of $\lrootG$ is
        \begin{equation}\label{cc3}
          \frac{(q^{n+1}-1)}{q-1}\frac{(q^{n-1}-1)}{q-1}.
        \end{equation}
\end{enumerate}	
\end{prop}

\section{The code $\cC(\Lambda_{\sigma})$ from the twisted embedding}
\label{twisted code}
In this section  we focus on the code $\cC(\Lambda_{\sigma}).$ We first recall  the construction of the codewords of $\cC(\Lambda_{\sigma}).$
Suppose $\Lambda_{\sigma}:=\{[X_1],[X_2],\dots, [X_N]\}\subseteq \PG({M}_{n+1}(q))$ and denote by ${M}_{n+1}^*(q)$ the dual of the vector space
  $M_{n+1}(q)$.
  For any functional ${\mathfrak m}\in {M}_{n+1}^*(q)$,
  there exists a unique matrix $M\in M_{n+1}(q)$ such that
  \[{\mathfrak m}\colon {M}_{n+1}(q)\rightarrow \FF_q,\quad \mathfrak{m}(X)=\Tr(XM)\]
  for all $X\in M_{n+1}(q)$.
  Define now $c_{\mathfrak m}$ as the $N$-uple
  \begin{equation}
    c_{\mathfrak{m}}=(\mathfrak{m} (X_1),\dots,\mathfrak{m}(X_N))
    \label{c_m}
  \end{equation}
  with   \begin{equation}\label{m_m} \mathfrak{m}(X_i)=\Tr(X_iM),\,\, 1\leq i\leq N, \end{equation}
  where $M\in {M}_{n+1}(q)$ is associated to $\mathfrak{m}$ as before.
  In this setting,
  \[\cC(\Lambda_{\sigma})=\{ c_{\mathfrak m} : \mathfrak{m}\in M_{n+1}^*(q) \}.\]
  It is clear that the codeword $c_{\mathfrak{m}}\in \cC(\Lambda_{\sigma})$ is associated to the matrix $M\in M_{n+1}(q)$ defining the functional $\mathfrak{m}$, so we shall denote it also by $c_M$.

\begin{lemma}
  \label{triv-lem}
  Let $X_1,\dots,X_N$ be representatives of the points $[X_1],\dots,[X_N]$
  of $\Lambda_{\sigma}$. Then the map
  \begin{equation}
    \label{triv-eq}
    ev:\begin{cases}
    M_{n+1}(q)\to\cC(\Lambda_\sigma) \\
    M \to c_M:=(\Tr(X_1M),\Tr(X_2M),\dots,\Tr(X_NM))
    \end{cases}
  \end{equation}
is a vector space isomorphism.
\end{lemma}
\begin{proof}
  By
  the properties of the trace, $ev$ is well defined, linear
  and surjective. Suppose now $M\in\ker(ev)$. Then
  $\Tr(X_iM)=0$ for all $i=1,\dots,N$. On the other hand
  $\langle\Lambda_{\sigma}\rangle=\PG(M_{n+1}(q))$, so we have that
  the representatives of the points of $\Lambda_{\sigma}$ contain
  a basis $(B_1,\dots,B_{(n+1)^2})$
  for $M_{n+1}(q)$. Since $\Tr(B_iM)=0$ for all $i=1,\dots,(n+1)^2$,
  it follows that it must be $\Tr(XM)=0$ identically for
  all $X\in M_{n+1}(q)$, whence $M=0$.
\end{proof}

\subsection{The weights of $\cC(\Lambda_{\sigma})$}\label{sec pesi}
To determine the weight of the codewords of $\cC(\Lambda_{\sigma})$ we need to compute the cardinality of $\Lambda_{\sigma}\cap [W]$ where $[W]$ is a hyperplane of $[\langle \Lambda_{\sigma}\rangle]=\PG(V\otimes V^*)$.
Recall from Section~\ref{Sec 1} that any hyperplane of $\PG(V\otimes V^*)$
can be regarded as the orthogonal subspace $[M^{\perp}]$ of a $(n+1)\times (n+1)$-matrix $M$ with respect to the saturation form $f$.
The following lemma, an extension of \cite[Lemma 3.2]{parte1}, is
essential.

\begin{lemma}\label{pesi}
 Let $[M^{\perp}]$ be a hyperplane of $\PG(V\otimes V^*)$ with $M\in M_{n+1}(q)$, and
    suppose $\sigma\in\Aut(\FF_q)$. Then
\begin{equation} |[M^{\perp}]\cap\Lambda_{\sigma}|=
  \frac{(q^{n+1}-1)(q^{n-1}-1)}{(q-1)^2}+\theta_M\cdot q^{n-1},
  \end{equation}
  where $\theta_M$ is the number of hyperplanes $[\xi]$ of $\PG(V)$ such that
  $[\xi]^{\sigma}\subseteq [\xi M]$.
\end{lemma}
  \begin{proof}
First observe that
   $\Lambda_{\sigma}=\{\bar{\vep}_{\sigma}(([x],[\xi]))\colon [x]\in [\xi]\}$ as a disjoint union
\[\Lambda_{\sigma}=\bigsqcup_{[\xi]\in\PG(V^*)}\{ [x^{\sigma}\otimes \xi] : [x]\in [\xi] \}.\]
So,
   \[ [M^{\perp}]\cap\Lambda_{\sigma}=\bigsqcup_{[\xi]\in\PG(V^*)}(\{ [x^{\sigma}\otimes\xi] : [x]\in [\xi] \}\cap [M^{\perp}]). \]

   By Proposition~\ref{prop},
   $[x^{\sigma}\otimes\xi]\in [M^{\perp}]$ if and only if $[x]^{\sigma}\in [\xi M]$.
Hence, $[x^{\sigma}\otimes \xi]\in \Lambda_{\sigma}\cap [M^{\perp}]$ if and only if $[x]^{\sigma}\in [\xi^{\sigma}]=[\xi]^{\sigma}$ and $[x]^{\sigma}\in [\xi M],$ i.e.
   \[ \Lambda_{\sigma}\cap [M^{\perp}]=\bigsqcup_{[\xi]\in\PG(V^*)}(\{ [x^{\sigma}\otimes \xi] : [x]^{\sigma}\in([\xi]^{\sigma}\cap [\xi M])\}). \]
 Turning to cardinalities,
 \begin{equation}
   \label{9bis} |\Lambda_{\sigma}\cap [M^{\perp}]|=\sum_{[\xi]\in\PG(V^*)}|[\xi]^{\sigma}\cap[\xi M]|.
 \end{equation}

Note that $[\xi]^{\sigma}$ is always a hyperplane of $\PG(V)$ and
$[\xi M]$ is a hyperplane of $\PG(V)$ if $\xi M$ is not the null vector of $V^*.$
The following cases may happen:
\begin{enumerate}[ {Case} 1.]
\item\label{c1} \framebox{$\xi\not\in\ker(M)$ and $[\xi M]\not=[\xi]^{\sigma}.$} In this case  $[\xi]^{\sigma}\cap[\xi M]$ is
a subspace of codimension $2$ of $\PG(V)$, so  $|[\xi]^{\sigma}\cap [\xi M]|=\frac{q^{n-1}-1}{q-1}$.
\item\label{c2} \framebox{$\xi\not\in\ker(M)$ and $[\xi M]=[\xi]^{\sigma}.$} In this case $|[\xi]^{\sigma}\cap [\xi M]|=\frac{q^{n}-1}{q-1}$.
\item\label{c3} \framebox{$\xi\in\ker(M).$} In this case $\xi M$ is the null functional, hence $[\xi M]=\PG(V).$ So, $[\xi]^{\sigma}\cap[\xi M]= [\xi]^{\sigma}$ is a hyperplane of $\PG(V)$ and again $|[\xi]^{\sigma}\cap [\xi M]|=\frac{q^{n}-1}{q-1}$.
\end{enumerate}
Clearly, Cases~\ref{c2} and \ref{c3} correspond to $[\xi]^{\sigma}\subseteq [\xi M]$ while Case~\ref{c1} corresponds to $[\xi]^{\sigma}\not\subseteq[\xi M]$. By plugging the values of the corresponding cardinalities in~\eqref{9bis}
we get
\begin{multline*}
  |\Lambda_{\sigma}\cap [M^{\perp}]|= \theta_M\cdot\frac{q^{n}-1}{q-1}+
  \left(\frac{q^{n+1}-1}{q-1}-\theta_M\right)\cdot\frac{q^{n-1}-1}{q-1}=\\
  \frac{(q^{n+1}-1)(q^{n-1}-1)}{(q-1)^2}+\theta_M\cdot q^{n-1}.
\end{multline*}
\end{proof}
\begin{remark}
  If $\sigma=1$, then the condition of Lemma~\ref{pesi} becomes
  $[\xi]\subseteq [\xi M]$, that is $\xi$ must be a \emph{left}
  eigenvector of $M$.
\end{remark}

The following is straightforward, considering that for any codeword $c_{\frak{m}}\in \cC(\Lambda_{\sigma})$, the weight of $c_{\frak{m}}$ is $wt(c_{\frak{m}})=N_{\sigma}-|[M^{\perp}]\cap\Lambda_{\sigma}|$, where $M$ is the matrix associated to the codeword  $c_{\frak{m}}$ as defined at the beginning of Section~\ref{twisted code}.
\begin{corollary}\label{weights}
Suppose $\sigma\in\Aut(\FF_q).$ The spectrum of the weights of $\cC(\Lambda_{\sigma})$	is
\[\{q^{n-1}\frac{(q^{n+1}-1)}{(q-1)} -q^{n-1} \theta_M\colon M\in M_{n+1}(q)\}\]
where $\theta_M$ is defined as in Lemma~\ref{pesi}.
	\end{corollary}

\begin{definition}\label{automorfismo notazione}
 Assume $\sigma \in \Aut(\FF_q)$ with $\sigma\neq\{1\}.$ Let  $\mathrm{Fix}(\sigma)\cong \FF_s$ be  the subfield  of  $\FF_q$ fixed by $\sigma.$ So, $q=s^t$ for some $t>1$ and there exists an index $j$ such that for all $x\in\FF_q$ we have $x^{\sigma}=x^{s^j}.$ 
 The \emph{$\sigma$-fixed subgeometry} of $\PG(n,q)$ is the point-line geometry having as points, the points of $\PG(n,q)$ fixed by $\sigma$, and as lines, the lines of $\PG(n,q)$ stabilized by $\sigma$; incidence is given by inclusion. 
\end{definition}

%
The  $\sigma$-fixed subgeometry  of $\PG(n,q)$ is isomorphic to $\PG(n,s).$

\subsection{Minimum weight codewords}
\label{minW}
Recall that
\begin{equation}\label{teta}
	\theta_M:=|\{ [\xi]\in\PG(V^*)\colon [\xi^{\sigma}]\subseteq [\xi M] \}|.
\end{equation}
By Corollary~\ref{weights}, in order to determine the minimum weight of $\cC(\Lambda_{\sigma})$ we need
to compute
\[ \max\{\theta_M: M\in M_{n+1}(q)\}. \]
We distinguish two cases.

\noindent {\bf{* Case A: $M$ is an invertible matrix.}} Hence, $[\xi M]$ is a hyperplane of $\PG(V)$ for all $[\xi]\in\PG(V^*)$. According to~\eqref{teta}, we need to count the number of points $[\xi]\in\PG(V^*)$ such that
 \begin{equation}
   \label{tce}
   [\xi^{\sigma}]=[\xi M].
 \end{equation}
This is equivalent to consider the set of fixed points of the semilinear collineation $\vartheta_{M,\sigma}$ of $\PG(V^*)$, induced by $\sigma$ and $M$, defined as
\begin{equation}\label{collineation}
\vartheta_{M,\sigma}: \begin{cases} \PG(V^*)\rightarrow \PG(V^*)\\
  [\zeta]\rightarrow [\zeta^{{\sigma}^{-1}} M].
\end{cases}
\end{equation}
Indeed, put $\xi:=\zeta^{\sigma^{-1}}.$ Then, $[\zeta]$ is fixed by $\vartheta_{M,\sigma}$, i.e. $[\zeta]=[\zeta^{{\sigma}^{-1}} M],$ if and only if
$[\xi^{\sigma}]=[\xi M]$. So, restricting to a subspace $W$ of $V^*$, the following is straightforward.
\begin{lemma} \label{l:c}
Let $W\subseteq V^*$ and $W^\sigma:=\{\xi^\sigma\colon \xi \in W\}.$ Then the number of points of $\PG(W^\sigma)$ fixed by $\vartheta_{M,{\sigma}}$ is the same as the number of points of $\PG(W)$ satisfying~\eqref{tce}, i.e.
\[|\{[\xi]\in \PG(W)\colon [\xi^\sigma]=[\xi M]\}|= |\{[\xi^\sigma]\in \PG(W^\sigma)\colon \vartheta_{M,\sigma}([\xi^\sigma])=[\xi^\sigma]\}|.\]
 \end{lemma}

 \begin{lemma}
   \label{l:I}
If $M=I$, then
$\theta_I=|\{ [\xi]\in\PG(V^*)\colon [\xi^{\sigma}]=[\xi]\}|= \frac{s^{n+1}-1}{s-1}.$
\end{lemma}
\begin{proof}
  By Lemma~\ref{l:c} we need to compute the number of fixed points
  of the semilinear collineation $\vartheta_{I,\sigma}.$
  A point $[\xi]\in\PG(V^*)$ is fixed by $\vartheta_{I,\sigma}$if and only if it is
  a point of the $\sigma$-fixed subgeometry $\mathrm{Fix}(\sigma)\cong \PG(n,s).$
  This completes the proof.
\end{proof}
 \begin{lemma}
  \label{c:2}
  Suppose that $M\in M_{n+1}(q)$ is invertible. The following hold.
 \begin{enumerate}
\item\label{c:2p1} If $W$ is a vector subspace of $V^*$ such that $[\xi^{\sigma}]=[\xi M]\,\,\forall \xi \in W$ then $\dim W \leq 1.$
\item\label{c:2p2} If $U$ is a subspace of $V^*$ with   $\dim(U)=2$ then $\PG(U)$ contains at most $s+1=|\FF_s|+1$ points $[\xi]$ such that $[\xi^{\sigma}]=[\xi M].$
\end{enumerate}

\end{lemma}

 \begin{proof}
Part~\ref{c:2p1}.  Suppose by way of contradiction that $\dim W\geq2$ and let $\xi_1,\xi_2$ be two linearly independent
  vectors of $W$
  such that $[\xi_i^{\sigma}]=[\xi_i M]$ for $i=1,2$. Consider the line $\ell=[\langle \xi_1,
  \xi_2\rangle]$ spanned by $[\xi_1]$ and $[\xi_2]$.  Then,
  \begin{equation}\label{e15} \xi_1^{\sigma}=\lambda_1\xi_1 M,\qquad
    \xi_2^{\sigma}=\lambda_2\xi_2 M
  \end{equation}
  for some $\lambda_1,\lambda_2\in \FF_q\setminus \{0\}$.

  We want to determine the number of points
  in $\ell$ satisfying~\eqref{tce} and different from $[\xi_1]$
  and $[\xi_2]$, i.e. the number of elements
  $\gamma\in\FF_q\setminus\{0\}$ for which
  there exists $\lambda_{\gamma}\neq0$ such that
  \begin{equation}
    \label{cc}
    (\xi_1+\gamma\xi_2)^{\sigma}=\lambda_{\gamma}(\xi_1+\gamma\xi_2)M.
  \end{equation}
  Observe that, by~\eqref{cc} and~\eqref{e15},
  \[ \lambda_{\gamma}(\xi_1+\gamma\xi_2)M= (\xi_1+\gamma\xi_2)^{\sigma}=\xi_1^{\sigma}+\gamma^{\sigma}\xi_2^{\sigma}=
    (\lambda_1\xi_1+\gamma^{\sigma}\lambda_2\xi_2)M.\]
  As $\xi_1M$ and $\xi_2M$ are linearly independent, it follows that
  $\lambda_{\gamma}=\lambda_1$ and $\gamma^{\sigma}\lambda_2=\lambda_1\gamma$.
  Hence
  \[{\gamma^{\sigma -1}} =\lambda_1\lambda_2^{-1}.\]
  Since the kernel of the group homomorphism $\FF_q^*\to\FF_q^*$ given
  by $x\to x^{\sigma-1}$ is   $\FF_s^*$,
  it is not difficult to see that the above equation (in the unknown $\gamma$) admits either $0$ or $s-1$ solutions.
  As $[\xi_1]$ and $[\xi_2]$  also satisfy~\eqref{tce}, we have that
  the overall number of points in $\ell$ satisfying~\eqref{tce} is at most
  $s-1+2=s+1$. Since $s<q$, this contradicts the hypothesis that
  all $[\xi]\in W$ are fixed. Part 1. is proved.\\
Part~\ref{c:2p2}. Note that $s+1\geq 3.$ If the number of points in $U$ satisfying ~\eqref{tce} is either $0$ or $1$ then  the thesis immediately follows. If the number of points in $U$ satisfying~\eqref{tce} is at least $2$ then we can repeat the same proof as in Part~\ref{c:2p1} to have the thesis.
\end{proof}

\begin{theorem}
\label{t:45}
  Let $W$ be a subspace of $V^*$ with $\dim(W)=r$ and suppose $M$ is invertible.
  Then the number of $[\xi]\in\PG(W)$ such that $[\xi^{\sigma}]=[\xi M]$ is at most  $\frac{s^r-1}{s-1}$.
\end{theorem}
\begin{proof}
Note that  if $M=I$, then the thesis follows from Lemma~\ref{l:I}.

We proceed by induction on $r.$
If $r\leq 2$ the result follows from Lemma~\ref{c:2}. Suppose $r>2.$ By Lemma~\ref{l:c}, we will count the number of fixed points of $\vartheta_{M,\sigma}$ in $\PG(W^\sigma).$
 Assume by induction that the number of fixed points of $\vartheta_{M,\sigma}$
  in any subspace of dimension $r'<r$ is at most $\frac{s^{r'}-1}{s-1}$.
  We distinguish three cases:
  \begin{enumerate}
    \item No hyperplane of $\PG(W^\sigma)$ is stabilized by $\vartheta_{M,\sigma}$. Hence, the set of fixed points by $\vartheta_{M,\sigma}$ spans a subspace of $\PG(W^\sigma)$  of codimension at least $2.$ The thesis follows directly by the inductive hypothesis.
    \item
      There is a hyperplane $[H]$ of $\PG(W^\sigma)$ stabilized
      by $\vartheta_{M,\sigma}$ and all fixed points of $\vartheta_{M,\sigma}$ are contained
      in $[H]$; then, by inductive hypothesis on $[H]$ the number of
      fixed points of $W^\sigma$ is at most $\frac{s^{r-1}-1}{s-1}\leq\frac{s^r-1}{s-1}$.
    \item
      There is a hyperplane $[H]$ stabilized by $\vartheta_{M,\sigma}$ and
      at least one point
      $[\xi^\sigma]\in\PG(W^\sigma)\setminus[H]$ fixed  by $\vartheta_{M,\sigma}.$
If $[\xi^\sigma]$ is the unique fixed point of $\PG(W^\sigma)$ then the thesis follows.
Otherwise, suppose $[\zeta^\sigma]\not=[\xi^\sigma]$ is a  point fixed by  $\vartheta_{M,\sigma}.$
      Then the line $\ell:=[\langle\zeta^\sigma,\xi^\sigma\rangle]$ is stabilized by $\vartheta_{M,\sigma}$
      and not contained in $[H]$, so the point  $\ell\cap[H]$ is fixed by $\vartheta_{M,\sigma}.$
      By
      inductive hypothesis, the number of fixed points in $[H]$ is at most
 $\frac{s^{r-1}-1}{s-1}$, so there exists at most $\frac{s^{r-1}-1}{s-1}$ such lines through
      $[\xi^\sigma]$.
      By Lemma~\ref{c:2}, every line of $\PG(W^\sigma)$ through
      $[\xi^\sigma]$ contains at most $s$ fixed points distinct from $[\xi^\sigma]$.
      So, the number of fixed points by $\vartheta_{M,\sigma}$ is at most
      \[ 1+s\frac{s^{r-1}-1}{s-1}=\frac{s^r-s+s-1}{s-1}=\frac{s^r-1}{s-1}. \]
    \end{enumerate}
  \end{proof}

  \noindent {\bf{* Case B: $M$ is an arbitrary matrix.}}

\begin{theorem} \label{m(r)}
  Let
  $m:{\mathbb N}\to{\mathbb N}$ be the function
    \begin{equation}
      \label{e:toc}
      m(r):=\frac{q^{n+1-r}-1}{q-1}+\frac{s^r-1}{s-1}.
    \end{equation}
    Then, for any matrix
    $M\in M_{n+1}(q)$ and $r=1,\dots,n+1$ we have
    \[ \max \{\theta_M: \rank(M)=r \}\leq m(r). \]
\end{theorem}
\begin{proof}
If $\xi\in\ker(M)$ then $[\xi M]=[0]=\PG(V^*)$, so the condition $[\xi^{\sigma}]\subseteq [\xi M]$ is always satisfied. Since $\rank(M)=r$, clearly $|[\ker(M)]|=\frac{q^{n+1-r}-1}{q-1}.$

Consider the following subspace of $V^*$:
\[U:=\langle \xi\in V^*\colon [\xi^{\sigma}]=[\xi M]\rangle=\]
\[ \langle \xi\in V^*\colon \lambda_\xi\xi^{\sigma}=\xi M, {\rm for\,\, some\,\, } 0\not=\lambda_\xi\in \FF_q \rangle.\]	

We claim that $U\cap\ker(M)=\mathbf{0}$.
Indeed, let ${\mathfrak B}:=(\xi_1,\dots,\xi_l)$ be a basis of $U$
consisting of vectors such that $[\xi_i^{\sigma}]=[\xi_i M]$.
Then, ${\mathfrak B}^{\sigma}=(\xi_1^{\sigma},\dots,\xi_l^{\sigma})$
is a basis of
$U^{\sigma}$ and $\xi_i M=\lambda_i\xi_i^{\sigma}$ for all $i=1,\dots,l$ and
$\lambda_i\in\FF_q\setminus \{0\}$.
Suppose $\zeta\in U\cap\ker(M).$ Then
$\zeta M=\mathbf{0}$ and $\zeta=\sum_{i=1}^l \alpha_i\xi_i$
for some $\alpha_i\in\FF_q$.
   It follows that
  \[ \mathbf{0}=\zeta M=\left(\sum_{i=1}^l \alpha_i\xi_i\right) M=\sum_{i=1}^l  \alpha_i (\xi_iM)=
    \sum_i\alpha_i\lambda_i\xi_i^{\sigma}.\]
  Since the vectors $\xi_i^{\sigma}$ are linearly independent, we have that
  $\alpha_1\lambda_1=\dots=\alpha_l\lambda_l=0$;
  as the $\lambda_i$'s are non-zero, this implies
  $\alpha_1=\dots=\alpha_l=0$, that is $\zeta=\mathbf{0}$ and so $U\cap\ker(M)=\{\mathbf{0}\}$.
  This implies $\dim(U)\leq r$. By Theorem~\ref{t:45} (clearly, $M$ restricted to $U$ is invertible) it follows that
  $\PG(U)$ contains at most $\frac{s^r-1}{s-1}$ points satisfying
  $[\xi^{\sigma}]=[\xi M]$.
  The thesis now follows.
\end{proof}
\begin{theorem}
  \label{mtt}
  The maximum number of points $[\xi]\in \PG(V^*)$ such that $[\xi^{\sigma}]\subseteq[\xi M]$ is
  \begin{enumerate}
  \item $m(3)=s^2+s+1$ if $\sigma^2=1$ and $n=2$; or
  \item $m(1)=\frac{q^{n}-1}{q-1}+1=q^{n-1}+\dots+2$,
    if $\sigma^2\neq 1$ or $n>2$.
  \end{enumerate}
\end{theorem}
\begin{proof}
  Recall that $q=s^t,\,t>1$
  \begin{itemize}
  \item
    If $n=2$, then the only possibilities for the rank of
    a non-null matrix $M\in M_3(q)$ are $r\in\{1,2,3\}$.
    So, by~\eqref{e:toc},  the possible values for $m(r)$ are
    \[ m(1)=\frac{q^2-1}{q-1}+1=q+2=s^t+2,\qquad
      m(2)=1+\frac{s^2-1}{s-1}=s+2,\]
    \[
      m(3)=\frac{s^3-1}{s-1}=s^2+s+1. \]
     If $t=2$ then $q=s^2$, i.e. $\sigma^2=1$ and
    $s^2+s+1>q+2=s^2+2>s+2$. 
    In this case $\theta_{\max}=\theta_I=m(3)=s^2+s+1$.
    \par
    If $t>2$, then
    $q+2=s^t+2\geq s^3+2> s^2+s+1>s+2$ and the maximum is $m(1)=q+2$.
    We can now check that if
    $M=\bee{11}$, then $\theta_M=\theta_{\max}=m(1)$.
  \item
    Suppose now $n>2$. We want to study the sequence
    \begin{equation} \label{std} m(1),\dots,
      m(r),\dots,m(n+1)\end{equation}
    of possible
  maxima as $r=\rank(M)$ ranges from $1$ to $n+1$; as such, we compare two
  successive terms in~\eqref{std} using~\eqref{e:toc}, namely
  \[ m(r):=\frac{q^{n+1-r}-1}{q-1}+\frac{s^r-1}{s-1}=
    \frac{s^{(n+1-r)t}-1}{s^t-1}+\frac{s^r-1}{s-1} \]
  and
\[
  m(r+1):=\frac{q^{n+1-r-1}-1}{q-1}+\frac{s^{r+1}-1}{s-1}=
  \frac{s^{(n-r)t}-1}{s^t-1}+\frac{s^{r+1}-1}{s-1}. \]
  Subtracting the first equation from
  the second we get
  \[ m(r+1)-m(r)=s^{r}-s^{t(n-r)}. \]
  This is non-positive if and only if
  $r\leq t(n-r)$, that is $r\leq\frac{tn}{t+1}$.
  This means that as $r$ grows
  the value of $m(r)$ is first decreasing and then increasing; so, the
  sequence $m(r)$ with $r=1,\dots,n+1$
  is convex and its maximum must be attained by values
  on the boundary of its domain, namely either for $r=1$ or $r=n+1$.
  In particular $m(1)=
  \frac{s^{tn}-1}{s^t-1}+1=s^{t(n-1)}+\dots+2$
  and
  $m(n+1)=\frac{s^{n+1}-1}{s-1}=s^{n}+\dots+1$.
  Comparing these two latter quantities, we see that
  $s^{tn-t}>s^{n}$ if $n>\frac{t}{t-1}$. However, $\frac{t}{t-1}\leq 2$.
  So, if $n>2$ the maximum of $m(r)$ is attained for $r=1$.
  We also have $\theta_{\bee{11}}=m(1)$, so this is actually the maximum
  for $\theta_{\max}$.
  \end{itemize}
\end{proof}
\begin{remark}
  As a consequence of Theorem~\ref{mtt} we see that for
  $r=1$ (with $n>2$ or $\sigma^2\not= 1$) or $r=3$ (with $n=2$ and $\sigma^2=1$)
  we have the equality
  \[ \max\{\theta_M: \rank(M)=r\}=m(r). \]
  We conjecture that this holds for all $r=1,\dots,n+1$.
\end{remark}

\subsection{Maximum weight codewords}
By Corollary~\ref{weights}, in order to determine the maximum  weight of $\cC(\Lambda_{\sigma})$ we need
to compute
\[ \min\{\theta_M: M\in M_{n+1}(q)\}. \]

In this section we provide sufficient conditions for
$\min\{\theta_M: M\in M_{n+1}(q)\}$ to be $0$.
When this happens, maximum weight codewords have weight
$q^{n-1}\frac{q^{n+1}-1}{q-1}$.

Observe that if $\rank(M)<n+1$, then $\theta_M>0$.
So we shall assume throughout the section that $M$ is an invertible
matrix. As seen in Lemma~\ref{l:c} with $V=W$, we have
$\theta_M=|\mathrm{Fix}(\vartheta_{M,\sigma})|$ where $\vartheta_{M,\sigma}$ is the collineation defined in~(\ref{teta})
$\vartheta_{M,\sigma}([\zeta])=[\zeta^{\sigma^{-1}}M]$.
So, to get $\theta_M=0$ we want to construct a $\sigma^{-1}$-semilinear collineation
which is fixed-point free.
\begin{lemma}
  Let $\sigma$ and $j$ be as in Definition~\ref{automorfismo notazione}.
  Suppose $\gcd(\frac{q^{n+1}-1}{q-1},s^j-1)>1$.
  Then, there exists
  an invertible matrix $M\in M_{n+1}(q)$ such that $\theta_M=0$.
 \end{lemma}
 \begin{proof}
   We need to show that there exists at least one
   $\sigma^{-1}$-semilinear collineation of $\PG(V^*)$ which is fixed-point-free. Observe that $\psi$ is a $\sigma^{-1}$-semilinear collineation
   which is fixed-point-free if and only if
   $\psi^{-1}$ is a $\sigma$-semilinear collineation which is
   also fixed-point-free. We shall now construct such
   $\psi^{-1}$.

   The vector space $V^*$ is isomorphic to the field $\FF_{q^{n+1}}$,
   regarded as a vector space over $\FF_q$.
   We identify the points $[\zeta]$ of $\PG(V^*)$ with elements
   of $\FF_{q^{n+1}}/\FF_{q}$.
   Let $\omega$ be a primitive element of $\FF_{q^{n+1}}$
   and define $\psi^{-1}:\PG(V^*)\to\PG(V^*)$ as the $\FF_q$-semilinear
   collineation
   $[x]\to [\zeta^{s^j}\omega]$.
   A point $[\zeta]\in\FF_{q^{n+1}}$ is
   fixed by $\psi^{-1}$ if and only if $\exists\lambda\in\FF_q$ with
   $\lambda\zeta=\zeta^{s^j}\omega$.
   This implies $\zeta^{s^j-1}\omega\in\FF_q$, that is
   $(\zeta^{s^j-1}\omega)^{q-1}=1$.
   In particular, the order of $\omega^{q-1}$ in $\FF_{q^{n+1}}^*$ must be the same
   as the order of $\zeta^{(s^j-1)(q-1)}$.
   Since $\omega$ is a primitive element of $\FF_{q^{n+1}}$,
   the order of $\omega^{q-1}$ is $(q^{n+1}-1)/(q-1)$.
   On the other hand,
   the order of $\zeta^{(s^j-1)(q-1)}$ divides
   \[ \frac{q^{n+1}-1}{\gcd((s^j-1)(q-1),q^{n+1}-1)}=
     \frac{(q^{n+1}-1)}{(q-1)\gcd(s^j-1,\frac{q^{n+1}-1}{q-1})}. \]
   It follows that if $\gcd(s^j-1,\frac{q^{n+1}-1}{q-1})>1$,
   then $\psi^{-1}$ does not have any fixed points.
   Since $\psi^{-1}$ is a semilinear collineation, it can be written as
   $\psi^{-1}([\zeta])=[\zeta^{\sigma}M^{-1}]$ for some invertible matrix $M$.
   It follows that also $\psi([\zeta])=[\zeta^{\sigma^{-1}}M]$ is fixed-point free;
   consequently  $\theta_M=0$. This completes the proof.
 \end{proof}
 \begin{corollary}
\label{co-odd}
   Suppose both $q$ and $n$ to be odd. Then there exists $M\in M_{n+1}(q)$ such
   that $\theta_M=0$.
 \end{corollary}
 \begin{proof}
   Under these assumptions both $s^j-1$ and $\frac{q^{n+1}-1}{q-1}$ are
   even. The previous lemma yields the result.
 \end{proof}
 As seen in Section~\ref{hyperplanes},
 when $n$ is odd and $\sigma^2=1$
 the geometry $\lrootG$ might admit semi-standard spread type hyperplanes
 $\cH_{\phi}$ arising from $\vep_{\sigma}$ of the form $\vep_{\sigma}^{-1}(M^{\perp})$.
 As the following lemma shows, these hyperplanes correspond to words of maximum weight.
 \begin{lemma}
\label{sph-w}
Suppose $\cH_{\phi}=\vep_{\sigma}^{-1}(M^{\perp})$ is a semi-standard spread type hyperplane of $\lrootG$ 
 where $\phi$ is a fixed-point-free semilinear involutory collineation of $\PG(V^*)$ defined by the matrix $M$.
 Then $c_M$ is a maximum weight codeword of $\cC(\Lambda_{\sigma}).$
\end{lemma}
\begin{proof}
  By Proposition~\ref{cardinalita' iperpiani quasi singolari},  $|\cH_{\phi}|=\frac{(q^{n+1}-1)(q^{n-1}-1)}{(q-1)^2}$.
  So, the weight of a codeword associated with $M^{\perp}$ is
  $q^{n-1}\frac{q^{n+1}-1}{q-1}$. It follows by
  Corollary~\ref{weights} that $\theta_M=0$. Hence $c_M$ is a maximum weight codeword of $\cC(\Lambda_{\sigma}).$
\end{proof}

Since $\theta_M\geq0$ by definition, in light of Corollary~\ref{weights}
the following is immediate.
\begin{corollary}
\label{co-mw}
  Suppose $\gcd(\frac{q^{n+1}-1}{q-1},s^j-1)>1$. Then
   the maximum weight of the codewords of $\cC(\Lambda_{\sigma})$
   is $w_{\max}=q^{n-1}\frac{q^{n+1}-1}{q-1}$.
   In particular, this happens for all $\sigma\in\Aut(\FF_q)$ when both $n$ and $q$ are odd.
 \end{corollary}

\subsection{Proof of Theorem~\ref{main thm 3}}\label{proof of main thm 3}
The computation of the length and dimension of $\cC(\Lambda_{\sigma})$ is straightforward from Section~\ref{Sec 1}. Indeed,  the length of $\cC(\Lambda_{\sigma})$ is the
number of point-hyperplane pairs $(p,H)$ of $\PG(n,q)$ with $p\in H,$ that is
\[N_{\sigma}=\frac{(q^{n+1}-1)(q^n-1)}{(q-1)^2}.\]
The dimension of $\cC(\Lambda_{\sigma})$ is the dimension of the embedding  $\bar{\vep}_{\sigma}$; so
\[ k_{\sigma}=(n+1)^2.\]

To determine the minimum distance, by Lemma~\ref{pesi} and Corollary~\ref{weights}, we need to find
the maximum of $\theta_M$ as $M$ varies in $M_{n+1}(q)$.
Hence, the minimum distance of $\cC(\Lambda_{\sigma})$ follows from Theorem~\ref{mtt} and we have \[ d_{\sigma}=\begin{cases}
	q^3-\sqrt{q}^3 & \text{ if }  \sigma^2=1 \text{ and } n=2, \\
	q^{2n-1}-q^{n-1} & \text{ if }  \sigma^2\neq 1 \text{ or } n>2.
\end{cases}
\].\hfill $\square$

\subsection{Proof of Theorem~\ref{main thm 4}}
\label{proof of main thm 4}
Point \ref{c4t0} of Theorem~\ref{main thm 4} follows from Proposition~\ref{c:min} and
Theorem~\ref{msub}.
Point \ref{pt4-13} of Theorem~\ref{main thm 4} follows from
Corollary~\ref{co-odd}, Corollary~\ref{co-mw} and Corollary~\ref{sph-w}. 

We remind that for $\sigma^2=1$ and $\sigma\neq 1$, the $\sigma$-norm function is defined as follows: $N\colon \FF_q\rightarrow \FF_s,\,\,N(x)=x^{s+1},$ where $\mathrm{Fix}(\sigma)=\FF_s,\,\,q=s^2.$

 \begin{theorem}
\label{mw:n2s1}
Suppose $n=2$, $\sigma^2=1$ and $\sigma\neq1$ and let $M\in M_{n+1}(q)\setminus\{0\}.$
The matrix $M$ defines a minimum weight codeword of
$\cC(\Lambda_{\sigma})$ if and only if there exist three linearly
independent vectors $\xi_1,\xi_2,\xi_3\in V^*$
and three non-null scalars $\alpha,\beta,\gamma\in\FF_q\setminus\{0\}$
such that
\begin{equation}
  \label{Cn2}
  \xi_1M=\alpha\xi_1^{\sigma},\qquad
  \xi_2M=\beta\xi_2^{\sigma},\qquad
  \xi_3M=\gamma\xi_3^{\sigma},\qquad
  N(\alpha)=N(\beta)=N(\gamma).
  \end{equation}
\end{theorem}
\begin{proof}
Suppose $[M^{\perp}]$ defines a minimum weight codeword of
$\cC(\Lambda_{\sigma}).$ By Theorem~\ref{mtt}, $M$ is an invertible $3\times 3$-matrix with $\theta_M=m(3)=s^2+s+1.$
Define $U:=\{\xi\in V^*\colon [\xi^{\sigma}]=[\xi M]\}.$ By the definition~(\ref{teta}) of $\theta_M$, we have $\theta_M=\frac{|U|-1}{q-1}.$

If $\dim(\langle U \rangle)\leq 2,$ then by Theorem~\ref{t:45} we have $\frac{|U|-1}{q-1}\leq s+1$, which contradicts the maximality of $\theta_M$ by Theorem~\ref{mtt}. Hence $U$ spans $V^*,$ i.e. there exist at least three linearly
independent vectors $\xi_1,\xi_2,\xi_3\in U$ and three scalars $\alpha,\beta,\gamma\in \FF_q\setminus\{0\}$ such that
  $\xi_1M=\alpha \xi_1^{\sigma}$,  $\xi_2M=\beta \xi_2^{\sigma}$,  $\xi_3M=\gamma \xi_3^{\sigma}.$ We now prove that $N(\alpha)=N(\beta)=N(\gamma).$

If no vector of $\langle U\rangle$ different from elements of $[\xi_i],\,\,i=1,2,3,$ is contained in $U$ then $\frac{|U|-1}{q-1}=|\{[\xi_1],[\xi_2],[\xi_3]\}|=3,$ which contradicts again the maximality of $\theta_M$ by Theorem~\ref{mtt}.
Hence there exists at least one vector $\zeta\in U$ which is not a multiple of any $\xi_i,\,\,i=1,2,3$.
So there is vector $\zeta\in U$ such that $[\zeta^{\sigma}]=[\zeta M]$ and
\[\zeta=a_1\xi_1+a_2\xi_2+a_3\xi_3\]
for some scalars $a_1,a_2,a_3,$ at least two of which are not null.
This implies that there exists  $\lambda\in \FF_q\setminus \{0\}$ such that
      \begin{equation}\label{sistema}
\left\{        \begin{array}{lll}
        a_1^{\sigma}=\lambda a_1\alpha\\
        a_2^{\sigma}=\lambda a_2\beta\\
        a_3^{\sigma}=\lambda a_3\gamma.\\
        \end{array}\right.
      \end{equation}
Denote now by $\Delta:=[
        \langle \xi_1,\xi_2\rangle]\cup [\langle \xi_1,\xi_3\rangle]
        \cup[\langle \xi_2,\xi_3\rangle] $ the triangle  of $\PG(V^*)$ whose sides are the lines spanned by $[\xi_i]$ and $[\xi_j],\,i\not=j,\,\,i,j\in \{1,2,3\}.$


Suppose by way of contradiction that all non-null vectors $\zeta\in U$ have the property that the associated projective points $[\zeta]$ are contained in $\Delta.$ Then, by Lemma~\ref{c:2}, we have
$\frac{|U|-1}{q-1}\leq 3s<s^2+s+1=m(3)$, which contradicts again the maximality of $\theta_M=\frac{|U|-1}{q-1}$ by Theorem~\ref{mtt}.

Hence there exists at least one non-null vector $\zeta\in U$  such that $[\zeta]\not\in\Delta$, i.e. there exist three not null scalars  $a_1,a_2,a_3$ and $\lambda\in \FF_q\setminus \{0\}$ satisfying~\eqref{sistema}.
This implies that $a_1a_2a_3\neq0.$  We can assume without
      loss of generality $a_1=1$.  Solving~\eqref{sistema} in $\lambda$, we get
      \begin{equation}
        \label{cond}
        \alpha^{-1}=a_2^{\sigma-1}\beta^{-1}=a_3^{\sigma-1}\gamma^{-1}.
      \end{equation}

        Let $\SSS=\{ x^{\sigma-1}: x\in\FF_q \}$ be the
     kernel of the norm function.
      If $\alpha^{-1}\beta\not\in\SSS$ or
      $\alpha^{-1}\gamma\not\in\SSS$
      we immediately see that there are no possible
      solutions to~\eqref{sistema} in the unknowns $a_2,a_3$   and this contradicts the existence of the scalars $a_2,a_3$ such that $\zeta\in U$. So,
      $\alpha^{-1}\beta\in\SSS$ and $\alpha^{-1}\gamma\in\SSS$, i.e., $N(\alpha^{-1}\beta)=1$ and $N(\alpha^{-1}\gamma)=1$ whence
$ N(\alpha)=N(\beta)=N(\gamma).$

Conversely, suppose there exist three linearly
independent vectors $\xi_1,\xi_2,\xi_3\in V^*$
and three non-null scalars $\alpha,\beta,\gamma\in\FF_q\setminus\{0\}$
such that~\eqref{Cn2} hold.  Then $M$ is invertible. We will show that $\theta_M=s^2+s+1$, hence the thesis would follow from  Theorem~\ref{mtt}.

Take an arbitrary vector $\zeta\in V^*$ as $\zeta=a_1\xi_1+a_2\xi_2+a_3\xi_3$ with $a_1,a_2,a_3\in \FF_q$. To compute $\theta_M$, we need to count the number of points $[\zeta]$ such that $\zeta\in U.$  This is equivalent to compute the number of solutions in the unknowns $a_1,a_2,a_3,\lambda$ of the system~\eqref{sistema}.
Put $\Delta:=\bigcup_{1\leq i< j\leq 3}\langle [\xi_i],[\xi_j]\rangle.$

Suppose $[\zeta]\not\in \Delta.$ Then $a_1a_2a_3\not= 0.$ Without loss of generality, put $a_1=1$, hence $\lambda=\alpha^{-1}$ and the system~\eqref{sistema} becomes

\begin{equation}\label{sistema-ter}
	\left\{        \begin{array}{lll}
		a_2^{\sigma -1}= \alpha^{-1} \beta\\
		a_3^{\sigma -1}=\alpha^{-1} \gamma.\\
	\end{array}\right.
\end{equation}
Since by hypothesis $N(\alpha^{-1} \beta)=N(\alpha^{-1} \gamma)=1$, the above system is solvable and each of the equations of~\eqref{sistema-ter} has exactly $s-1$ solutions.
So, there exist precisely $(s-1)^2=s^2-2s+1$ points $[\zeta]\in \PG(V^*)\setminus \Delta$ with $\zeta\in U.$

Suppose $[\zeta]\in \Delta\setminus\{[\xi_1],[\xi_2],[\xi_3]\}.$   Without loss of generality, assume $[\zeta]\in \langle [\xi_2],[\xi_3]\rangle$. i.e. $a_1=0$ and take $a_2=1.$ Then the system~\eqref{sistema} becomes
\begin{equation}\label{sistema-quater}
			a_3^{\sigma -1}=\beta^{-1} \gamma.
	\end{equation}
Since $N(\beta^{-1} \gamma)=1$, the above equation is solvable and has exactly $s-1$ solutions.

Thus, there exist precisely $(s-1)$ points $[\zeta]\not\in \{[\xi_2],[\xi_3]\}$ on the line $[\langle \xi_2, \xi_3\rangle]$ with $\zeta\in U.$
The same argument applies to the lines $[\langle\xi_1,\xi_2\rangle]$ and $[\langle\xi_1,\xi_3\rangle]$. So,
\[\frac{|U|-1}{q-1}=\theta_M= (s-1)^2+3(s-1)+3=s^2+s+1.\]
By Theorem~\ref{mtt}, $\theta_M=m(3).$ So, the theorem is proved.
\end{proof}
Part~\ref{c4t1} of Theorem~\ref{main thm 4} is exactly Theorem~\ref{mw:n2s1} and Part 3 of Theorem~\ref{main thm 4} is the following theorem.
\begin{theorem}
\label{mw:n>2}
  Suppose $n>2$ or $\sigma^2\neq1$ and let $M\in M_{n+1}(q)\setminus\{0\}.$
  The matrix $M$ defines a minimum weight codeword of
  $\cC(\Lambda_{\sigma})$ if and only if ${\bar{\vep}_{\sigma}}^{-1}(M^{\perp})$ is a quasi-singular but not a singular hyperplane of $\lrootG$.

  The hyperplane $[M^{\perp}]$ defines a second minimum weight codeword of
  $\cC(\Lambda_{\sigma})$ if and only if ${\bar{\vep}_{\sigma}}^{-1}(M^{\perp})$ is a singular hyperplane of $\lrootG$.
 \end{theorem}
\begin{proof}
Suppose $\cH$ is a quasi-singular hyperplane of $\lrootG$.  By Proposition~\ref{qsh}, $\cH={\bar{\vep}_{\sigma}}^{-1}(M^{\perp}),$ with $M$ a non-null $(n+1)$-matrix of rank $1$ and $|\cH|=\frac{(q^{n+1}-1)(q^{n-1}-1)}{(q-1)^2}+\frac{q^n-1}{q-1} q^{n-1}$ if $\cH$ is a singular hyperplane or $|\cH|=\frac{(q^{n+1}-1)(q^{n-1}-1)}{(q-1)^2}+(\frac{q^n-1}{q-1} +1)q^{n-1}$ if $\cH$ is a quasi-singular but not singular hyperplane of $\lrootG$ by Proposition~\ref{cardinalita' iperpiani quasi singolari}. By the injective property of the embedding $\bar{\vep}_{\sigma}$, we have $|[M^{\perp}]\cap \Lambda_{\sigma}|=|{\bar{\vep}_{\sigma}}^{-1}(M^{\perp})|.$ If we consider the codeword $c_M\in \cC(\Lambda_{\sigma})$ defined by the matrix $M$, we know by Corollary~\ref{pesi} that the weight of $c_M$ is $wt(c_M)=N_{\sigma}-|[M^{\perp}]\cap \Lambda_{\sigma}|=N_{\sigma}-|\cH|.$

By Theorem~\ref{mtt} and  Corollary~\ref{weights}, comparing the cardinalities of a singular hyperplane and a quasi-singular non-singular hyperplane of $\lrootG$, we immediately see that the minimum weight codewords are associated to a quasi-singular non-singular hyperplane $\vep_{\sigma}^{-1}(M^{\perp})$ of $\lrootG$ with $\theta_M=\frac{q^n-1}{q-1}+1$.
As the singular hyperplanes of $\lrootG$ yield codewords with $\theta_M=\frac{q^n-1}{q+1}$, they correspond to
codewords with the second lowest weight.

Conversely, let  $[M^{\perp}]$ define a minimum weight codeword of
$\cC(\Lambda_{\sigma})$;
then $\theta_M=m(1)=\frac{q^n-1}{q-1} +1$ and, by Theorem~\ref{mtt}, $M$ is  a rank $1$ matrix.
 By Proposition~\ref{qsh}, ${\bar{\vep}_{\sigma}}^{-1}(M^{\perp})$ is a quasi-singular hyperplane of $\lrootG$.
By Lemma~\ref{pesi}, the weight of the minimum weight codewords corresponds  to the maximal cardinality of  the set $[M^{\perp}]\cap \Lambda_{\sigma}.$ Since $|[M^{\perp}]\cap \Lambda_{\sigma}|=|{\bar{\vep}_{\sigma}}^{-1}(M^{\perp})|$,  the thesis now follows from Proposition~\ref{cardinalita' iperpiani quasi singolari}.

Suppose now $[M^{\perp}]$ defines a second lowest weight codeword of
$\cC(\Lambda_{\sigma})$. Then, by Theorem~\ref{mtt} and Corollary~\ref{weights}, its  weight corresponds to $\theta_M=m(1)-1=\frac{q^n-1}{q-1}.$ We
first show that $M$ has rank $1$.

By way of contradiction suppose $\rank(M)\geq 2.$ Then by Theorem~\ref{m(r)}, $\theta_M\leq m(2)$. By Theorem~\ref{m(r)},  $m(2)=\frac{q^{n-1}-1}{q-1} +\frac{s^2-1}{s-1}$ and $m(1)-1=\frac{q^n-1}{q-1},$ but this contradicts $(m(1)-1)=\theta_M\leq m(2).$
 Hence $\rank(M)=1$. The thesis now follows by Proposition~\ref{qsh} and Proposition~\ref{cardinalita' iperpiani quasi singolari}.
   \end{proof}

As a consequence of Theorem~\ref{mw:n>2} we point out that the minimum weight codewords and the second minimum weight codewords of $\cC(\Lambda_{\sigma})$  depend on the properties of the hyperplanes of  $\lrootG$ and not on the way it is embedded into a projective space.

\subsection{Proof of Theorem~\ref{main thm 5}}
\label{proof of main thm 5}



Relying on the representation of the codewords of $\cC(\Lambda_{\sigma})$ given by~\eqref{triv-eq}
  of Lemma~\ref{triv-lem}, we see that $\varrho(g)$ maps an
  element of $\cC(\Lambda_{\sigma})$ to an element of $\cC(\Lambda_{\sigma})$, by the rule
  \begin{multline*}
    c_M=(\Tr(X_1M),\dots,\Tr(X_NM))\to \\
    c_{g^{-1}Mg^{\sigma}}=
    (\Tr(X_1g^{-1}Mg^{\sigma}),\dots,\Tr(X_Ng^{-1}Mg^{\sigma})).
  \end{multline*}
  Furthermore,
  $\varrho(g)$ is a linear transformation of $\cC(\Lambda_{\sigma})$
  since
  \begin{multline*}
    \varrho(g)(\alpha c_A+\beta c_B)=\varrho(g)(c_{\alpha A+\beta B})= \\
    =\varrho(g)(\Tr(X_1(\alpha A+\beta B)),\dots,\Tr(X_N(\alpha A+\beta B)))= \\
    =(\Tr(X_1g^{-1}(\alpha A+\beta B)g^{\sigma}),\dots,\Tr(X_Ng^{-1}(\alpha A+\beta B)g^{\sigma}))=\\
    \alpha(\Tr(X_1g^{-1}Ag^{\sigma}),\dots,\Tr(X_Ng^{-1}Ag^{\sigma}))+\\
    \beta (\Tr(X_1g^{-1}Bg^{\sigma}),\dots,\Tr(X_Ng^{-1}Bg^{\sigma}))=\\
    =\alpha c_{g^{-1}Ag^{\sigma}}+\beta c_{g^{-1}Bg^{\sigma}}=
    \alpha \varrho(g)(c_A)+
    \beta \varrho(g)(c_B).
  \end{multline*}

  There remains to prove that $\varrho(g)$ is an isometry. To this
  purpose observe that by the cyclic property of the trace,
  \[ \Tr(X_i g^{-1}Mg^{\sigma})=\Tr(g^{\sigma}X_ig^{-1} M), \]
  for all $i=1,\dots,N$ and all $M\in M_{n+1}(q)$.
  On the other hand,
  $[g^{\sigma}X_ig^{-1}]=[X_i]^g=[X_{\pi(g)(i)}]$ by~\eqref{action-L};
  consequently, there is a function
  $\lambda:\GL(n+1,q)\times\{1,\dots,N\}\to\FF_q^{\star}$
  depending on $g$ and $i$ such that $g^{\sigma}X_ig^{-1}=\lambda(g,i) X_{\pi(g)(i)}$.
  Thus, we can write
  \[
    \Tr(X_ig^{-1}Mg^{\sigma})=\lambda(g,i)\Tr(X_{\pi(g)(i)}M),
  \]
  for all $i=1,\dots,N$.
  In particular, the components of $c_M^g$ are obtained by permuting
  and multiplying by the scalars $\lambda(g,i)$ the components of $c_M$.
  Since $\lambda(g,i)\neq0$, the number of null components in
  $c_M^g$, i.e. the number of $i$ such that
  $\Tr(X_ig^{-1}Mg^{\sigma})=0$, is the same as the number
  of null components in $c_M$.
  This proves that the weight of $c_M^g$ is the same as
  that of $c_M$, i.e. that $\varrho(g)$ is a isometry.

  To conclude, observe that $g^{-1}Mg^{\sigma}=M$ for all $M$ implies,
  as a special case when $M=I$, that $g^{-1}g^{\sigma}=1$, i.e.
  $g^{\sigma}=g$. This happens if and only if $g$ is a matrix
  with entries over $\FF_{s}$. On the other hand,
  $Mg=g^{-1}M$ for all $M$ implies that $g$ must be a scalar matrix.
  It follows that the kernel of the action is $K_{\sigma}$. \hfill $\square$
\begin{remark}
  The action of $\GL(n+1,q)$ on $M_{n+1}(q)$ described in Theorem~\ref{main thm 5} is
  not the action of $\GL(n+1,q)$ on $\Lambda_{\sigma}$ which lifts through $\bar{\vep}_{\sigma}$ which is instead $X^g:=g^{\sigma}Xg^{-1}$.
 These two actions are adjoint with respect to the saturation form,
  in the sense that if $f$ is the saturation form,
  $f(X^g,M)=\Tr(g^{\sigma}Xg^{-1}M)=\Tr(Xg^{-1}Mg^{\sigma})=f(X,M^g)$.

\end{remark}

Authors' addresses:
\vskip.4cm\noindent\nobreak
\begin{minipage}[t]{8cm}
\small{Ilaria Cardinali, \\
Dep. Information Engineering and Mathematics \\University of Siena\\
Via Roma 56, I-53100 Siena, Italy\\
ilaria.cardinali@unisi.it }
\end{minipage}
\hfill
\begin{minipage}[t]{6cm}
\small{Luca Giuzzi\\
D.I.C.A.T.A.M. \\
Universit\`a di Brescia\\
Via Branze 43, I-25123 Brescia, Italy \\
luca.giuzzi@unibs.it}
\end{minipage}

\end{document}